\documentclass{amsart}
\usepackage{graphicx}
\newtheorem{thm}{Theorem}
\newtheorem{cor}[thm]{Corollary}
\newtheorem{lem}[thm]{Lemma}
\newtheorem{prop}[thm]{Proposition}
\theoremstyle{definition}

\theoremstyle{remark}

\newcommand{\norm}[1]{\left\Vert#1\right\Vert}
\newcommand{\abs}[1]{\left\vert#1\right\vert}

\newcommand{\R}{\mathbf R}
\newcommand{\C}{\mathbf C}
\newcommand{\eps}{\varepsilon}

\begin{document}

\title[A Note on Homogeneous Functions and Parallelogram Law]{A Note on Homogeneous Functions and Parallelogram Law}%
\author{Wenhan Wang}%
\address{}%
\email{wangwh@math.washington.edu}%
\author{Wen Wang}%
\address{}%
\email{wwang@math.wsu.edu}%

\thanks{}%

\begin{abstract}
In this article, we investigate a new characterization of the parallelogram law in a normed linear space. We give equivalent conditions to the parallelogram law, in terms of the homogeneous property of a continuous positive definite function on the normed space.
\end{abstract}
\maketitle

\

Let $K=\R$ or $\C$ be the base field, and suppose $X$ be a normed $K$-linear space, with norm $\norm{\cdot}$. A continuous function $p:X\to\R$ is called \emph{positive definite} if $p(x)\geq 0$ for all $x\in X$, and $p(x)=0$ if and only if $x=0\in X$. Moreover, $p:X\to\R$ is called homogeneous, if for any $\lambda\in K$, $p(\lambda x)=\abs{\lambda}p(x)$. We first note the following lemma on the comparison of two positive definite homogeneous continuous functions on $X$.

\

\begin{lem}
Let $p:X\to\R$ be a continuous homogeneous function on $X$. If $p$ is bounded, then $p\equiv 0$ on $X$.
\end{lem}
\begin{proof}
Since $p$ is bounded, we may assume that the $\abs{p(x)}\leq M$ for all $x\in X$. It then follows that for any $x\in X$ and any $\eps>0$,
$$\abs{p(x)}=\eps\cdot\abs{p\left(\frac{x}{\eps}\right)}\leq\eps\cdot M.$$
Since $\eps>0$ is arbitrary, it is necessary that $\abs{p(x)}=0$. Hence $p(x)=0$.
\end{proof}

\

We then obtain immediately the following corollary on the comparison of two homogeneous functions $p_1(x)$ and $p_2(x)$.

\

\begin{cor}
Let $p_1$ and $p_2$ be two homogeneous functions on $X$, such that $p_1-p_2$ is bounded on $X$. Then $p_1=p_2$.
\end{cor}
\begin{proof}
Apply the above Lemma to $p=p_1-p_2$.
\end{proof}

\

We then introduce and prove the main Theorem on characterization of the parallelogram law in terms of homogeneous functions. For any $y\in X$ with $\norm{y}=1$, we define a function
\begin{equation}
p_y(x):=\sqrt{\frac{1}{2}\left(\norm{x+y}^2+\norm{x-y}^2\right) - 1}.
\end{equation}
Note that $p_y:X\to\R$ is well-defined. To see this, note that
$$\norm{x+y}^2+\norm{x-y}^2\geq \frac{1}{2}\left(\norm{x+y}+\norm{x-y}\right)^2
\geq\frac{1}{2}\norm{2y}^2=2,$$
hence $\frac{1}{2}\left(\norm{x+y}^2+\norm{x-y}^2\right) - 1\geq 0$, and $p_y$ is well-defined from $X$ to $\R$. Moreover, following a closer look at the above arguments, we have

\

\begin{prop}
For any $y\in X$, $p_y$ is a continuous positive definite function.
\end{prop}

\

Before we introduce the main Theorem, we shall compare $p_y(x)$ with the norm $\norm{x}$.

\

\begin{lem}
For any $x\in X$, with $\norm{x}\geq 1$, we have $\norm{x}-1\leq p_y(x)\leq\norm{x}+1$.
\end{lem}
\begin{proof}
To show that $p_y(x)\geq\norm{x}-1$, note that
\begin{eqnarray*}
p_y(x) &=& \sqrt{\frac{1}{2}\left(\norm{x+y}^2+\norm{x-y}^2\right)-1} \\
&\geq& \sqrt{\frac{1}{4}\left(\norm{x+y}+\norm{x-y}\right)^2-1} \\
&\geq& \sqrt{\frac{1}{4}\norm{2x}^2-1} \\
&\geq& \sqrt{\norm{x}^2-1} \\
&\geq& \sqrt{\norm{x}^2-2\norm{x}+1} \\
&=& \norm{x}-1.
\end{eqnarray*}
To prove the other inequality, note that
\begin{eqnarray*}
p_y(x) &=& \sqrt{\frac{1}{2}\left(\norm{x+y}^2+\norm{x-y}^2\right)-1} \\
&\leq& \sqrt{\norm{x}^2+2\norm{x}\norm{y}+\norm{y}^2-1} \\
&=& \sqrt{\norm{x}^2+2\norm{x}} \\
&<& \sqrt{\norm{x}^2+2\norm{x}+1} \\
&=& \norm{x}+1.
\end{eqnarray*}
\end{proof}

\

Now we may give the main Theorem, on the characterization of parallelogram rule in terms of $p_y(x)$, for all $y\in X$.

\

\begin{thm}
Let $X$ be a normed linear space, and for each $y\in X$, $p_y(x)$ is defined as above. Then the following statements are equivalent.
\begin{enumerate}
\item[(i)] For all $y\in X$, $p_y$ is homogeneous.
\item[(ii)] For all $y\in X$, $p_y(x)=\norm{x}$.
\item[(iii)] $\norm{\cdot}$ is induced from an inner product on $X$.
\end{enumerate}
\end{thm}
\begin{proof}
(i)$\Rightarrow$(ii). Note that both $p_y$ and the norm $\norm{\cdot}$ are homogeneous functions on $X$. The above lemma shows that $p_y(x)-\norm{x}$is bounded on $X$. Hence $p_y(x)=\norm{x}$ for all $x\in X$. \\
(ii)$\Rightarrow$(iii). Note that $p_y(x)=\norm{x}$ gives that for any $y\in X$ with $\norm{y}=1$, $2\norm{x}^2+2\norm{y}^2=\norm{x+y}^2+\norm{x-y}^2$. In general, if $y\neq0$, then this also implies the parallelogram law, by taking $\tilde y=\frac{y}{\norm{y}}$. \\
(iii)$\Rightarrow$(i). If $\norm{\cdot}$ is induced from an inner product $\langle\cdot,\cdot\rangle:X\times X\to\C$, then we may prove directly that $p_y(x)=\langle x,x\rangle$, which is thus homogeneous.
\end{proof}

\bibliographystyle{amsplain}

\end{document}